\documentclass[11pt]{article}
\usepackage{geometry}                
\usepackage{amsmath,amssymb}
\usepackage{amsthm}
\usepackage{hyperref}
\usepackage{graphicx,epsfig,color}
\usepackage{booktabs,bm,multirow}
\usepackage{cases}

\newtheorem{theorem}{Theorem}

\newtheorem{remark}[theorem]{Remark}

\input{rdkrtk.sty}

\title{Randomized double and triple Kaczmarz for solving extended normal equations}
\author{Kui Du\thanks{School of Mathematical Sciences and Fujian Provincial Key Laboratory of Mathematical Modelling and High Performance Scientific Computing, Xiamen University, Xiamen 361005, China (kuidu@xmu.edu.cn).},\quad  Xiao-Hui Sun\thanks{School of Mathematical Sciences, Xiamen University, Xiamen 361005, China ({sunxh@stu.xmu.edu.cn}).}} 
\date{}                                           

\begin{document}
\maketitle

\begin{abstract} The randomized Kaczmarz algorithm has received considerable attention recently because of its simplicity, speed, and the ability to approximately solve large-scale linear systems of equations. In this paper we propose randomized double and triple Kaczmarz algorithms to solve extended normal equations of the form $\bf A^\top Ax=A^\top b-c$. The proposed algorithms avoid forming $\bf A^\top A$ explicitly and work for {\it arbitrary} $\mbf A\in\mbbr^{m\times n}$  (full rank or rank deficient, $m\geq n$ or $m<n$). {\it Tight} upper bounds showing exponential convergence in the mean square sense of the proposed algorithms are presented and numerical experiments are given to illustrate the theoretical results.  
\vspace{.5mm} 



\end{abstract}

\section{Introduction}

We consider the following extended normal equations \beq\label{pne} \mbf A^\top\mbf A\mbf x=\mbf A^\top \mbf b-\mbf c\eeq with {\it arbitrary} $\mbf A\in\mbbr^{m\times n}$ (full rank or rank deficient, $m\geq n$ or $m<n$), $\mbf b\in\mbbr^m$, and $\mbf c\in\mbbr^n$. The linear system (\ref{pne}) arises in some applications, such as multilevel Levenberg-Marquardt methods for training artificial neural networks \cite{calandra2019appro} or Fletcher's exact penalty function approach \cite{fletcher1972class}. Because of the existence of $\mbf c$, standard methods for least squares problems can not be used directly. Conjugate gradient-type methods  based on {\it full} matrix-vector multiplications for solving (\ref{pne}) with full column rank $\mbf A$ have been proposed recently in \cite{calandra2020itera}. However, these methods are {\it not feasible} when full matrix-vector multiplications are not available or ``expensive'' to obtain (e.g., the data matrix $\mbf A$ is dynamically growing or so large that it does not fit in computer memory). 

In recent years, randomized iterative algorithms for solving large-scale linear systems or linear least squares problems have been greatly developed due to low memory footprints (these methods do not need to load the entire coefficient matrix into memory, and each iteration only requires a sample of rows and/or columns) and good numerical performance, such as the randomized Kaczmarz (RK) algorithm \cite{strohmer2009rando}, the  randomized coordinate descent algorithm \cite{leventhal2010rando}, and their extensions, e.g., \cite{zouzias2013rando,ma2015conve,needell2015rando,bai2018greed,bai2018relax,bai2019parti,ma2018itera,necoara2019faste,zhang2019new,du2020rando,moorman2020rando,niu2020greed,wu2020proje,zhang2020relax}. In this paper, we propose two variants of the randomized Kaczmarz algorithm \cite{strohmer2009rando} to solve the extended normal equations (\ref{pne}). More specifically, we propose a randomized double Kaczmarz (RDK) algorithm for a solution of the linear system (\ref{pne}) if it is consistent ($\mbf c\in\ran(\mbf A^\top)$) and a randomized triple Kaczmarz (RTK) algorithm for a least squares solution of the linear system (\ref{pne}) if it is inconsistent ($\mbf c\notin\ran(\mbf A^\top)$). We make no assumptions about the dimensions or rank of $\mbf A$. We present {\it tight} upper bounds for the exponential convergence in the mean square sense of the proposed algorithms. 

The organization of this paper is as follows. In the rest of this section, we give  notation and preliminary. In Section 2, we review the RK algorithm. In Section 3 we describe the RDK algorithm and the RTK algorithm, and we also establish their convergence theory. In Section 4 we report the numerical results to illustrate the theoretical results. Finally, we present brief concluding remarks in Section 5.

{\it Notation and preliminary}. For any random variable $\bm\xi$, we use $\mbbe\bem\bm\xi\eem$ to denote the expectation of $\bm\xi$. For an integer $m\geq 1$, let $[m]:=\{1,2,3,\ldots,m\}$. For any vector $\mbf b\in\mbbr^m$, we use $\mbf b_i$, $\bf b^\top$ and $\|\mbf b\|_2$ to denote the $i$th entry, the transpose and the Euclidean norm of $\mbf b$, respectively. We use $\mbf I$ to denote the identity matrix whose order is clear from the context. For any matrix $\mbf A\in\mbbr^{m\times n}$, we use $\mbf A_{i,:}$, $\mbf A_{:,j}$ $\mbf A^\top$, $\mbf A^\dag$, $\|\mbf A\|_2$, $\|\mbf A\|_\rmf$, $\ran(\mbf A)$, $\rank(\mbf A)$, $\sigma_{\rm max}(\mbf A)$ and $\sigma_{\rm min}(\mbf A)$ to denote the $i$th row, the $j$th column, the transpose, the Moore-Penrose pseudoinverse, the 2-norm, the Frobenius norm, the column space, the rank, the maximum and the minimum nonzero singular values of $\mbf A$, respectively. All the convergence results depend on the positive number $\rho$ defined as $$\rho:=1-\frac{\sigma_{\rm min}^2(\mbf A)}{\|\mbf A\|_\rmf^2}.$$ For any nonzero matrix $\mbf A$ and any $\mbf u\in\ran(\mbf A^\top)$, it holds \beq\label{lem}\mbf u^\top\l(\mbf I-\frac{\mbf A^\top \mbf A}{\|\mbf A\|_\rmf^2}\r)\mbf u\leq\rho \|\mbf u\|_2^2.\eeq 

\section{Randomized Kaczmarz} 
 In each iteration, the RK algorithm orthogonally projects the current estimate vector onto the affine hyperplane defined by a randomly chosen row of $\bf Ax=b$. See Algorithm 1 for details. Theorem \ref{rk1} shows that the sequence $\{\mbf x^k\}_{k=0}^\infty$ in the RK algorithm with arbitrary initial vector $\mbf x^0\in\mbbr^n$ for a consistent linear system $\bf Ax=b$ converges to $\mbf x_\star^0=(\mbf I-\mbf A^\dag\mbf A)\mbf x^0+\mbf A^\dag\mbf b$, which is the orthogonal projection of $\mbf x^0$ onto the solution set $\{\mbf x\in\mbbr^n\ |\ \bf Ax = b\}$. We emphasize that we make no assumptions about the dimensions or rank of $\mbf A$. The proof of Theorem \ref{rk1} can be found in, e.g., \cite{du2019tight,zouzias2013rando,necoara2019faste}. For completeness and clarity, we provide a proof. 
\begin{center}
\begin{tabular*}{160mm}{l}
\toprule {\bf Algorithm 1:} RK for $\bf Ax=b$\\ 
\hline \noalign{\smallskip}
\qquad Initialize $\mbf x^0\in\mbbr^n$\\
\qquad {\bf for} $k=1,2,\ldots,$ {\bf do}\\
\qquad \qquad  Pick $i\in[m]$ with probability ${\|\mbf A_{i,:}\|^2_2}/{\|\mbf A\|_\rmf^2}$\\
\qquad \qquad  Set $\dsp\mbf x^k=\mbf x^{k-1}-\frac{\mbf A_{i,:}\mbf x^{k-1}-\mbf b_i}{\|\mbf A_{i,:}\|_2^2} (\mbf A_{i,:})^\top$\\
\bottomrule
\end{tabular*}
\end{center}

\begin{theorem}\label{rk1} Suppose that $\mbf b\in\ran(\mbf A)$ $($i.e., $\bf Ax=b$ is consistent$)$. The sequence $\{\mbf x^k\}_{k=0}^\infty$ in the {\rm RK} algorithm with arbitrary $\mbf x^0\in\mbbr^n$ satisfies \beq\label{rk}\mbbe\bem\|\mbf x^k-\mbf x_\star^0\|_2^2\eem\leq\rho^k\|\mbf x^0-\mbf x_\star^0\|_2^2,\eeq where $\mbf x_\star^0=(\mbf I-\mbf A^\dag\mbf A)\mbf x^0+\mbf A^\dag\mbf b$.
\end{theorem}  
\begin{proof} By $\mbf A\mbf x_\star^0=\mbf b$, we have \begin{align} \mbf x^k-\mbf x_\star^0 &=\mbf x^{k-1}-\mbf x_\star^0-\frac{\mbf A_{i,:}\mbf x^{k-1}-\mbf b_i}{\|\mbf A_{i,:}\|_2^2} (\mbf A_{i,:})^\top\nn \\ &=\mbf x^{k-1}-\mbf x_\star^0-\frac{\mbf A_{i,:}\mbf x^{k-1}-\mbf A_{i,:}\mbf x_\star^0}{\|\mbf A_{i,:}\|_2^2} (\mbf A_{i,:})^\top\nn \\ &= \l(\mbf I-\frac{(\mbf A_{i,:})^\top\mbf A_{i,:}}{\|\mbf A_{i,:}\|_2^2}\r)(\mbf x^{k-1}-\mbf x_\star^0)\label{up}.\end{align} 
It follows that $$\|\mbf x^k-\mbf x_\star^0\|_2^2=(\mbf x^{k-1}-\mbf x_\star^0)^\top\l(\mbf I-\frac{(\mbf A_{i,:})^\top\mbf A_{i,:}}{\|\mbf A_{i,:}\|_2^2}\r)(\mbf x^{k-1}-\mbf x_\star^0).$$
Taking condition expectation gives $$\mbbe\bem\|\mbf x^k-\mbf x_\star^0\|_2^2\ |\mbf x^{k-1}\eem=(\mbf x^{k-1}-\mbf x_\star^0)^\top\l(\mbf I-\frac{\mbf A^\top\mbf A}{\|\mbf A\|_\rmf^2}\r)(\mbf x^{k-1}-\mbf x_\star^0).$$ Noting that $\mbf x^0-\mbf x_\star^0=\mbf A^\dag(\mbf A\mbf x^0-\mbf b)\in\ran(\mbf A^\top)$ and $(\mbf A_{i,:})^\top\mbf A_{i,:}(\mbf x^{k-1}-\mbf x_\star^0)\in\ran(\mbf A^\top)$, by (\ref{up}), we can show that $\mbf x^k-\mbf x_\star^0\in\ran(\mbf A^\top)$ by induction. Then by (\ref{lem}), we have $$\mbbe\bem\|\mbf x^k-\mbf x_\star^0\|_2^2\ |\mbf x^{k-1}\eem\leq\rho\|\mbf x^{k-1}-\mbf x_\star^0\|_2^2.$$ By the law of total expectation we have $$\mbbe\bem\|\mbf x^k-\mbf x_\star^0\|_2^2\eem\leq\rho\mbbe\bem\|\mbf x^{k-1}-\mbf x_\star^0\|_2^2\eem\leq\cdots\leq\rho^k\|\mbf x^0-\mbf x_\star^0\|_2^2.$$ This completes the proof.
\end{proof}

\begin{remark} If $\sigma_{\rm max}(\mbf A)=\sigma_{\rm min}(\mbf A)$, then the inequality $(\ref{lem})$ becomes equality. This yields that all the inequalities in the proof of Theorem $\ref{rk1}$ become equalities. Therefore, the convergence bound in Theorem $\ref{rk1}$ is tight. 
\end{remark}


\section{Algorithms and main results} 

\subsection{The RDK algorithm for the case $\mbf c\in\ran(\mbf A^\top)$} The randomized extended Kaczmarz (REK) algorithm \cite{zouzias2013rando} solves $\bf A^\top A\mbf x=\mbf A^\top \mbf b$ via intertwining an iterate of RK on $\bf A^\top z=0$ with an iterate of RK on $\bf A x=b-z$. More precisely, the $k$th iterate of the REK algorithm, $\mbf x^k$, is the iterate of RK on ${\bf A x=b-z}^k$ from $\mbf x^{k-1}$, where $\mbf z^k$ is the $k$th iterate of RK on $\bf A^\top z=0$ with $\mbf z^0\in \mbf b+\ran(\mbf A)$. Inspired by the REK algorithm, we propose Algorithm 2 to solve the problem (\ref{pne}) for the case $\mbf c\in\ran(\mbf A^\top)$. We note that $\mbf z^k$ in Algorithm 2 is the $k$th iterate of RK on $\bf A^\top z=c$ with $\mbf z^0\in \mbf b+\ran(\mbf A)$, and $\mbf x^k$ is the iterate of RK on ${\bf A x=b-z}^k$ from $\mbf x^{k-1}$ with arbitrary $\mbf x^0\in\mbbr^n$. Since two RK iterates are used in each iteration of Algorithm 2, we call it a randomized double Kaczmarz (RDK) algorithm. By (\ref{rk}), we have \beq\label{rkz}\mbbe\bem\|\mbf z^k-\mbf z_\star^0\|_2^2\eem\leq\rho^k\|\mbf z^0-\mbf z_\star^0\|_2^2,\eeq where $$\mbf z_\star^0=(\mbf I-\mbf A\mbf A^\dag)\mbf z^0+(\mbf A^\top)^\dag\mbf c=(\mbf I-\mbf A\mbf A^\dag)\mbf b+(\mbf A^\top)^\dag\mbf c.$$ We show that the sequence $\{\mbf x^k\}_{k=0}^\infty$ in the RDK algorithm converges to a solution of the linear system  (\ref{pne}) in Theorem \ref{rk2}. We emphasize that we make no assumptions about the dimensions or rank of $\mbf A$.

\begin{center}
\begin{tabular*}{160mm}{l}
\toprule {\bf Algorithm 2:} RDK for $\bf A^\top Ax =A^\top b- c$ with $\mbf c\in\ran(\mbf A^\top)$ \\ 
\hline \noalign{\smallskip}
\qquad Initialize $\mbf z^0\in \mbf b+\ran(\mbf A)$ and $\mbf x^0\in\mbbr^n$\\
\qquad {\bf for} $k=1,2,\ldots,$ {\bf do}\\
\qquad \qquad  Pick $j\in[n]$ with probability ${\|\mbf A_{:,j}\|^2_2}/{\|\mbf A\|_\rmf^2}$\\
\qquad \qquad  Set $\dsp\mbf z^k=\mbf z^{k-1}-\frac{(\mbf A_{:,j})^\top\mbf z^{k-1}-\mbf c_j}{\|\mbf A_{:,j}\|_2^2} \mbf A_{:,j}$\\
\qquad \qquad  Pick $i\in[m]$ with probability ${\|\mbf A_{i,:}\|^2_2}/{\|\mbf A\|_\rmf^2}$\\
\qquad \qquad  Set $\dsp\mbf x^k=\mbf x^{k-1}-\frac{\mbf A_{i,:}\mbf x^{k-1}-\mbf b_i+\mbf z_i^k}{\|\mbf A_{i,:}\|_2^2} (\mbf A_{i,:})^\top$\\
\bottomrule
\end{tabular*}
\end{center}

\begin{theorem}\label{rk2} Suppose that $\mbf c\in\ran(\mbf A^\top)$ $($i.e., the linear system $(\ref{pne})$ is consistent$)$. The sequence $\{\mbf x^k\}_{k=0}^\infty$ in the {\rm RDK} algorithm with $\mbf z^0\in \mbf b+\ran(\mbf A)$ and arbitrary $\mbf x^0\in\mbbr^n$ satisfies $$\mbbe\bem\|\mbf x^k-\mbf x_\star^0\|_2^2\eem\leq\frac{k\rho^k}{\|\mbf A\|_\rmf^2}\|\mbf z^0-\mbf z_\star^0\|_2^2+\rho^k\|\mbf x^0-\mbf x_\star^0\|_2^2,$$ where $\mbf z_\star^0=(\mbf I-\mbf A\mbf A^\dag)\mbf b+(\mbf A^\top)^\dag\mbf c$, and $\mbf x_\star^0=(\mbf I-\mbf A^\dag\mbf A)\mbf x^0+\mbf A^\dag\mbf b-(\mbf A^\top\mbf A)^\dag\mbf c$ is a solution of {\rm(\ref{pne})}.
\end{theorem}
\begin{proof} Let \beq\label{xhat}\wh{\mbf x}^k={\bf x}^{k-1}-\frac{\mbf A_{i,:}({\bf x}^{k-1}-\bf A^\dag b+(A^\top A)^\dag c)}{\|\mbf A_{i,:}\|_2^2}(\mbf A_{i,:})^\top.\eeq
By ${\bf A(A^\top A)^\dag=(\mbf A^\top)^\dag}$, we have \begin{align}{\mbf x}^k-\wh{\mbf x}^k & =\frac{\mbf b_i-\mbf A_{i,:}\mbf A^\dag\mbf b+\mbf A_{i,:} (\mbf A^\top \mbf A)^\dag \mbf c-\mbf z_i^k}{\|\mbf A_{i,:}\|_2^2}(\mbf A_{i,:})^\top\nn \\ & = \frac{\mbf I_{i,:}((\mbf I-\mbf A\mbf A^\dag)\mbf b+(\mbf A^\top)^\dag \mbf c-\mbf z^k)}{\|\mbf A_{i,:}\|_2^2}(\mbf A_{i,:})^\top\nn \\ &=\frac{\mbf I_{i,:}(\mbf z_\star^0-\mbf z^k)}{\|\mbf A_{i,:}\|_2^2}(\mbf A_{i,:})^\top \label{eq1} \end{align}
 and by $\mbf A(\mbf I-\mbf A^\dag\mbf A)\mbf x^0=\mbf 0$,  we have \begin{align}\wh{\mbf x}^k-\mbf x_\star^0 &= {\bf x}^{k-1}-\mbf x_\star^0-\frac{\mbf A_{i,:}({\bf x}^{k-1}-(\mbf I-\mbf A^\dag\mbf A)\mbf x^0-\bf A^\dag b+(A^\top A)^\dag c)}{\|\mbf A_{i,:}\|_2^2}(\mbf A_{i,:})^\top \nn \\ &= {\bf x}^{k-1}-\mbf x_\star^0-\frac{\mbf A_{i,:}({\bf x}^{k-1}-\mbf x_\star^0)}{\|\mbf A_{i,:}\|_2^2}(\mbf A_{i,:})^\top\nn\\ & = \l(\mbf I-\frac{(\mbf A_{i,:})^\top\mbf A_{i,:}}{\|\mbf A_{i,:}\|_2^2}\r)({\bf x}^{k-1}-\mbf x_\star^0).\label{eq2}\end{align} 
 By the orthogonality $(\wh{\mbf x}^k-{\mbf x_\star^0})^\top({\mbf x}^k-\wh{\mbf x}^k)=0$ (which is obvious from (\ref{eq1}) and (\ref{eq2})), we have \beq\label{ksum}\|{\mbf x}^k-{\mbf x_\star^0}\|_2^2=\|{\mbf x}^k-\wh{\mbf x}^k\|_2^2+\|\wh{\mbf x}^k-{\mbf x_\star^0}\|_2^2.\eeq 
 Let $\mbbe_{k-1}\bem\cdot\eem$ denote the conditional expectation given the first $k-1$ iterations of RDK. 
 Let $\mbbe_{k-1}^i\bem\cdot\eem$ denote the expectation with respect to the $k$th row chosen and  $\mbbe_{k-1}^j\bem\cdot\eem$ denote the expectation with respect to the $k$th column chosen. Then by the law of total expectation we have $\mbbe_{k-1}\bem\cdot\eem=\mbbe_{k-1}^j\bem\mbbe_{k-1}^i\bem\cdot\eem\eem$.
 It follows from \begin{align*}\mbbe_{k-1}\bem\|{\mbf x}^k-\wh{\mbf x}^k\|_2^2\eem 
 &= \mbbe_{k-1}\bem\dsp\frac{(\mbf I_{i,:}(\mbf z_\star^0-\mbf z^k))^2}{\|\mbf A_{i,:}\|_2^2}\eem =\mbbe_{k-1}^j\bem\mbbe_{k-1}^i\bem\dsp\frac{(\mbf I_{i,:}(\mbf z_\star^0-\mbf z^k))^2}{\|\mbf A_{i,:}\|_2^2}\eem\eem\\
 &=\mbbe_{k-1}^j\bem\dsp\frac{\|\mbf z^k-\mbf z_\star^0\|_2^2}{\|\mbf A\|_\rmf^2}\eem =\frac{1}{\|\mbf A\|_\rmf^2}\mbbe_{k-1}\bem\|\mbf z^k-\mbf z_\star^0\|_2^2\eem\end{align*} that 
 \beq\label{ksum1}\mbbe\bem\|{\mbf x}^k-\wh{\mbf x}^k\|_2^2\eem = \frac{1}{\|\mbf A\|_\rmf^2}\mbbe\bem\|\mbf z^k-\mbf z_\star^0\|_2^2\eem
    \leq \frac{\rho^k}{\|\mbf A\|_\rmf^2}\|\mbf z^0-\mbf z_\star^0\|_2^2. \quad (\mbox{by (\ref{rkz})})\eeq 
 By $\mbf x^0-\mbf x_\star^0\in\ran(\mbf A^\top)$, it is easy to show that $\mbf x^{k-1}-\mbf x_\star^0\in\ran(\mbf A^\top)$ by induction. It follows from \begin{align*}\mbbe_{k-1}\bem\|\wh{\mbf x}^k-\mbf x_\star^0\|_2^2\eem &=\mbbe_{k-1}\bem(\wh{\mbf x}^k-\mbf x_\star^0)^\top(\wh{\mbf x}^k-\mbf x_\star^0)\eem\\&=\mbbe_{k-1}\bem\dsp({\bf x}^{k-1}-\mbf x_\star^0)^\top\l(\mbf I-\frac{(\mbf A_{i,:})^\top\mbf A_{i,:}}{\|\mbf A_{i,:}\|_2^2}\r)^2({\bf x}^{k-1}-\mbf x_\star^0)\eem\\&=\mbbe_{k-1}\bem\dsp({\bf x}^{k-1}-\mbf x_\star^0)^\top\l(\mbf I-\frac{(\mbf A_{i,:})^\top\mbf A_{i,:}}{\|\mbf A_{i,:}\|_2^2}\r)({\bf x}^{k-1}-\mbf x_\star^0)\eem\\&=\dsp({\bf x}^{k-1}-\mbf x_\star^0)^\top\l(\mbf I-\frac{\bf A^\top A}{\|{\bf A}\|_\rmf^2}\r)({\bf x}^{k-1}-\mbf x_\star^0)\\&\leq\rho\|{\bf x}^{k-1}-\mbf x_\star^0\|_2^2\quad (\mbox{by }(\ref{lem}))\end{align*} that \beq\label{ksum2}\mbbe\bem\|\wh{\mbf x}^k-\mbf x_\star^0\|_2^2\eem\leq\rho\mbbe\bem\|{\bf x}^{k-1}-\mbf x_\star^0\|_2^2\eem.\eeq 
Combining (\ref{ksum}), (\ref{ksum1}), and (\ref{ksum2}) yields \begin{align*}\mbbe\bem\|{\mbf x}^k-\mbf x_\star^0\|_2^2\eem &=\mbbe\bem\|{\mbf x}^k-\wh{\mbf x}^k\|_2^2\eem+\mbbe\bem\|\wh{\mbf x}^k-\mbf x_\star^0\|_2^2\eem\\&\leq \frac{\rho^k}{\|\mbf A\|_\rmf^2}\|\mbf z^0-\mbf z_\star^0\|_2^2+\rho\mbbe\bem\|{\bf x}^{k-1}-\mbf x_\star^0\|_2^2\eem\\&\leq\frac{2\rho^k}{\|\mbf A\|_\rmf^2}\|\mbf z^0-\mbf z_\star^0\|_2^2+\rho^2\mbbe\bem\|{\bf x}^{k-2}-\mbf x_\star^0\|_2^2\eem\\&\leq\cdots\\ &\leq\frac{k\rho^k}{\|\mbf A\|_\rmf^2}\|\mbf z^0-\mbf z_\star^0\|_2^2+\rho^k\|{\bf x}^0-\mbf x_\star^0\|_2^2.
\end{align*} It is trivial to verify that $\mbf x_\star^0$ is a solution of (\ref{pne}). This completes the proof.
\end{proof}
\begin{remark} If $\sigma_{\rm max}(\mbf A)=\sigma_{\rm min}(\mbf A)$, then the inequalities $(\ref{lem})$ and $(\ref{rk})$ become equalities. This yields that all the inequalities in the proof of Theorem $\ref{rk2}$ become equalities. Therefore, the convergence bound in Theorem $\ref{rk2}$ is tight. 
\end{remark}
 
\subsection{The RTK algorithm for the case $\mbf c\notin\ran(\mbf A^\top)$}

Given $\mbf U\in\mbbr^{m\times k}$, $\mbf V\in\mbbr^{k\times n}$ and $\mbf y\in\mbbr^m$, the REK-RK algorithm \cite[Algorithm 2]{ma2018itera} solves the factorized linear system $\bf UVx=y$ for the case $\mbf y\notin\ran(\mbf U)$ via intertwining an iterate of REK for solving $\bf U^\top U z = U^\top y$ with an iterate of RK on $\bf Vx =z$. Inspired by the REK-RK algorithm, we propose Algorithm 3 for the linear system (\ref{pne}) with $\mbf c \notin \ran(\mbf A^\top)$. We note that $\mbf y^k$ in Algorithm 3 is the $k$th iterate of RK on $\bf A y=0$ with $\mbf y^0\in \mbf c+\ran(\mbf A^\top)$, $\mbf z^k$ is the iterate of RK on ${\bf A^\top z=c-y}^k$ from $\mbf z^{k-1}$ with $\mbf z^0\in \mbf b+\ran(\mbf A)$, and $\mbf x^k$ is the iterate of RK on ${\bf Ax=b-z}^k$ from $\mbf x^{k-1}$ with arbitrary $\mbf x^0\in\mbbr^n$. Since three RK iterates are used in each iteration of Algorithm 3, we call it a randomized triple Kaczmarz (RTK) algorithm. Actually, $\mbf y^k$ and $\mbf z^k$ of RTK are exactly the iterates of  RDK applied for the system $\bf AA^\top z=Ac$ (or $\bf Ay=0$ and $\bf A^\top z=c-y$).  By  Theorem \ref{rk2}, we have \beq\label{thm1}\mbbe\bem\|\mbf z^k-\mbf z_\star^0\|_2^2\eem\leq\frac{k\rho^k}{\|\mbf A\|_\rmf^2}\|\mbf y^0-\mbf y_\star^0\|_2^2+\rho^k\|\mbf z^0-\mbf z_\star^0\|_2^2,\eeq where $\mbf y_\star^0=(\mbf I-\mbf A^\dag\mbf A)\mbf c$ and $\mbf z_\star^0=(\mbf I-\mbf A\mbf A^\dag)\mbf b+(\mbf A^\top)^\dag\mbf c$. We show that the sequence $\{\mbf x^k\}_{k=0}^\infty$ in the RTK algorithm converges to a least squares solution of (\ref{pne}) in Theorem \ref{rk3}. We emphasize that we make no assumptions about the dimensions or rank of $\mbf A$.

\begin{center}
\begin{tabular*}{160mm}{l}
\toprule {\bf Algorithm 3:} RTK for $\bf A^\top Ax =A^\top b- c$ with $\mbf c\notin\ran(\mbf A^\top)$ \\ 
\hline \noalign{\smallskip}
\qquad Initialize $\mbf y^0\in\mbf c+\ran(\mbf A^\top)$, $\mbf z^0\in \mbf b+\ran(\mbf A)$, and $\mbf x^0\in\mbbr^n$\\
\qquad {\bf for} $k=1,2,\ldots,$ {\bf do}\\
\qquad \qquad  Pick $l\in[m]$ with probability ${\|\mbf A_{l,:}\|^2_2}/{\|\mbf A\|_\rmf^2}$\\
\qquad \qquad  Set $\dsp\mbf y^k=\mbf y^{k-1}-\frac{\mbf A_{l,:}\mbf y^{k-1}}{\|\mbf A_{l,:}\|_2^2} (\mbf A_{l,:})^\top$\\
\qquad \qquad  Pick $j\in[n]$ with probability ${\|\mbf A_{:,j}\|^2_2}/{\|\mbf A\|_\rmf^2}$\\
\qquad \qquad  Set $\dsp\mbf z^k=\mbf z^{k-1}-\frac{(\mbf A_{:,j})^\top\mbf z^{k-1}-\mbf c_j+\mbf y_j^k}{\|\mbf A_{:,j}\|_2^2} \mbf A_{:,j}$\\
\qquad \qquad  Pick $i\in[m]$ with probability ${\|\mbf A_{i,:}\|^2_2}/{\|\mbf A\|_\rmf^2}$\\
\qquad \qquad  Set $\dsp\mbf x^k=\mbf x^{k-1}-\frac{\mbf A_{i,:}\mbf x^{k-1}-\mbf b_i+\mbf z_i^k}{\|\mbf A_{i,:}\|_2^2} (\mbf A_{i,:})^\top$\\
\bottomrule
\end{tabular*}
\end{center}

\begin{theorem}\label{rk3} Suppose that $\mbf c\notin\ran(\mbf A^\top)$ $($i.e., the linear system $(\ref{pne})$ is inconsistent$)$. The sequence $\{\mbf x^k\}_{k=0}^\infty$ in the {\rm RTK} algorithm with $\mbf y^0\in\mbf c+\ran(\mbf A^\top)$, $\mbf z^0\in \mbf b+\ran(\mbf A)$, and arbitrary $\mbf x^0\in\mbbr^n$ satisfies $$\mbbe\bem\|\mbf x^k-\mbf x_\star^0\|_2^2\eem\leq\frac{k(k+1)\rho^k}{2\|\mbf A\|_\rmf^4}\|\mbf y^0-\mbf y_\star^0\|+\frac{k\rho^k}{\|\mbf A\|_\rmf^2}\|\mbf z^0-\mbf z_\star^0\|_2^2+\rho^k\|\mbf x^0-\mbf x_\star^0\|_2^2,$$ where $\mbf y_\star^0=(\mbf I-\mbf A^\dag\mbf A)\mbf c$, $\mbf z_\star^0=(\mbf I-\mbf A\mbf A^\dag)\mbf b+(\mbf A^\top)^\dag\mbf c$, and $ \mbf x_\star^0=(\mbf I-\mbf A^\dag\mbf A)\mbf x^0+\mbf A^\dag\mbf b-(\mbf A^\top\mbf A)^\dag\mbf c$  is a least squares solution of {\rm(\ref{pne})}.
\end{theorem}
\begin{proof} Let $\wh{\mbf x}^k$ be the vector given in (\ref{xhat}). We note that the equalities (\ref{eq1})--(\ref{ksum}), and the inequality (\ref{ksum2}) in the proof of Theorem \ref{rk2} still hold. By (\ref{thm1}), the estimate (\ref{ksum1}) becomes \beq\label{ksum3} \mbbe\bem\|\mbf x^k-\wh{\mbf x}^k\|_2^2\eem= \frac{1}{\|\mbf A\|_\rmf^2}\mbbe\bem\|\mbf z^k-\mbf z_\star^0\|_2^2\eem\leq\frac{k\rho^k}{\|\mbf A\|_\rmf^4}\|\mbf y^0-\mbf y_\star^0\|_2^2+\frac{\rho^k}{\|\mbf A\|_\rmf^2} \|\mbf z^0-\mbf z_\star^0\|_2^2.\eeq  
Combining (\ref{ksum}), (\ref{ksum2}), and (\ref{ksum3}) yields \begin{align*}\mbbe\bem\|{\mbf x}^k-\mbf x_\star^0\|_2^2\eem &=\mbbe\bem\|{\mbf x}^k-\wh{\mbf x}^k\|_2^2\eem+\mbbe\bem\|\wh{\mbf x}^k-\mbf x_\star^0\|_2^2\eem\\&\leq \frac{k\rho^k}{\|\mbf A\|_\rmf^4}\|\mbf y^0-\mbf y_\star^0\|_2^2+\frac{\rho^k}{\|\mbf A\|_\rmf^2} \|\mbf z^0-\mbf z_\star^0\|_2^2+\rho\mbbe\bem\|{\bf x}^{k-1}-\mbf x_\star^0\|_2^2\eem\\&\leq\frac{k\rho^k}{\|\mbf A\|_\rmf^4}\|\mbf y^0-\mbf y_\star^0\|_2^2+\frac{(k-1)\rho^k}{\|\mbf A\|_\rmf^4}\|\mbf y^0-\mbf y_\star^0\|_2^2\\ &\quad+\frac{2\rho^k}{\|\mbf A\|_\rmf^2} \|\mbf z^0-\mbf z_\star^0\|_2^2+\rho^2\mbbe\bem\|{\bf x}^{k-2}-\mbf x_\star^0\|_2^2\eem
\\&\leq
\cdots
\\&\leq\frac{k(k+1)\rho^k}{2\|\mbf A\|_\rmf^4}\|\mbf y^0-\mbf y_\star^0\|_2^2+\frac{k\rho^k}{\|\mbf A\|_\rmf^2}\|\mbf z^0-\mbf z_\star^0\|_2^2+\rho^k\|{\bf x}^0-\mbf x_\star^0\|_2^2.
\end{align*} It is trivial to verify that $\mbf x_\star^0$ is a least squares solution of (\ref{pne}). Then we  complete the proof.\end{proof}
\begin{remark}
If $\sigma_{\rm max}(\mbf A)=\sigma_{\rm min}(\mbf A)$, then the inequalities $(\ref{lem})$ and $(\ref{thm1})$ become equalities. This yields that all the inequalities in the proof of Theorem $\ref{rk3}$ become equalities. Therefore, the convergence bound in Theorem $\ref{rk3}$ is tight.	
\end{remark}

\section{Numerical results}
In this section, we report the numerical results of the RDK algorithm and the RTK algorithm for solving (\ref{pne}). The purpose is to illustrate our theoretical results (Theorems \ref{rk2} and \ref{rk3}) via simple examples. All experiments are performed using MATLAB on a laptop with 2.7-GHz Intel Core i7 processor, 16-GB memory, and Mac operating system. 

The matrix $\mbf A$ and the vectors $\mbf b$ and $\mbf c$ in (\ref{pne}) are generated by using the MATLAB functions {\tt diag}, {\tt null}, {\tt ones}, {\tt qr}, {\tt rand}, and {\tt randn} as follows. Given $m$, $n$, $r = \rank(\mbf A)$, and $\kappa\geq 1$, we construct the matrix $\mbf A$ by $\bf A = UDV^\top$, where $\mbf U\in\mbbr^{m\times r}$, $\mbf D\in\mbbr^{r\times r}$ and $\mbf V\in\mbbr^{n\times r}$ are given by {\tt [U,$\sim$]=qr(randn(m,r),0)}, {\tt D=diag(ones(r,1)+($\kappa$-1)*rand(r,1))} and {\tt [V,$\sim$]=qr(randn(n,r),0)}. So the condition number of $\mbf A$, which is defined as $\sigma_{\rm max}(\mbf A)/\sigma_{\rm min}(\mbf A)$, is upper bounded by $\kappa$. The vector $\mbf b$ is taken to be {\tt b=randn(m,1)}. For the case $\mbf c\in\ran(\mbf A^\top)$, the vector $\mbf c$ is constructed by {\tt c=A'*randn(m,1)}. For the case $\mbf c\notin\ran(\mbf A^\top)$, the vector $\mbf c$ is constructed by {\tt c=randn(n,1)+null(A)*randn(n-r,1)}.

In all experiments we use $\mbf y^0=\mbf c$, $\mbf z^0=\mbf b$, and $\mbf x^0 = \mbf 0$. In Figures \ref{fig1} and \ref{fig2} we plot the error $\|\mbf x^k-\mbf A^\dag\mbf b+(\mbf A^\top\mbf A)^\dag \mbf c\|_2^2$ (average of 50 independent trials) of RDK and RTK. For all cases, RDK and REK converge. In particular, for $\kappa=1$, which means all nonzero singular values of $\mbf A$ are the same, the convergence bounds in Theorems \ref{rk2} and \ref{rk3} are attained (see Figure \ref{fig1}). All these experimental results support the theoretical findings presented in Theorems \ref{rk2} and \ref{rk3}.

\begin{figure}[htb]
\centerline{\epsfig{figure=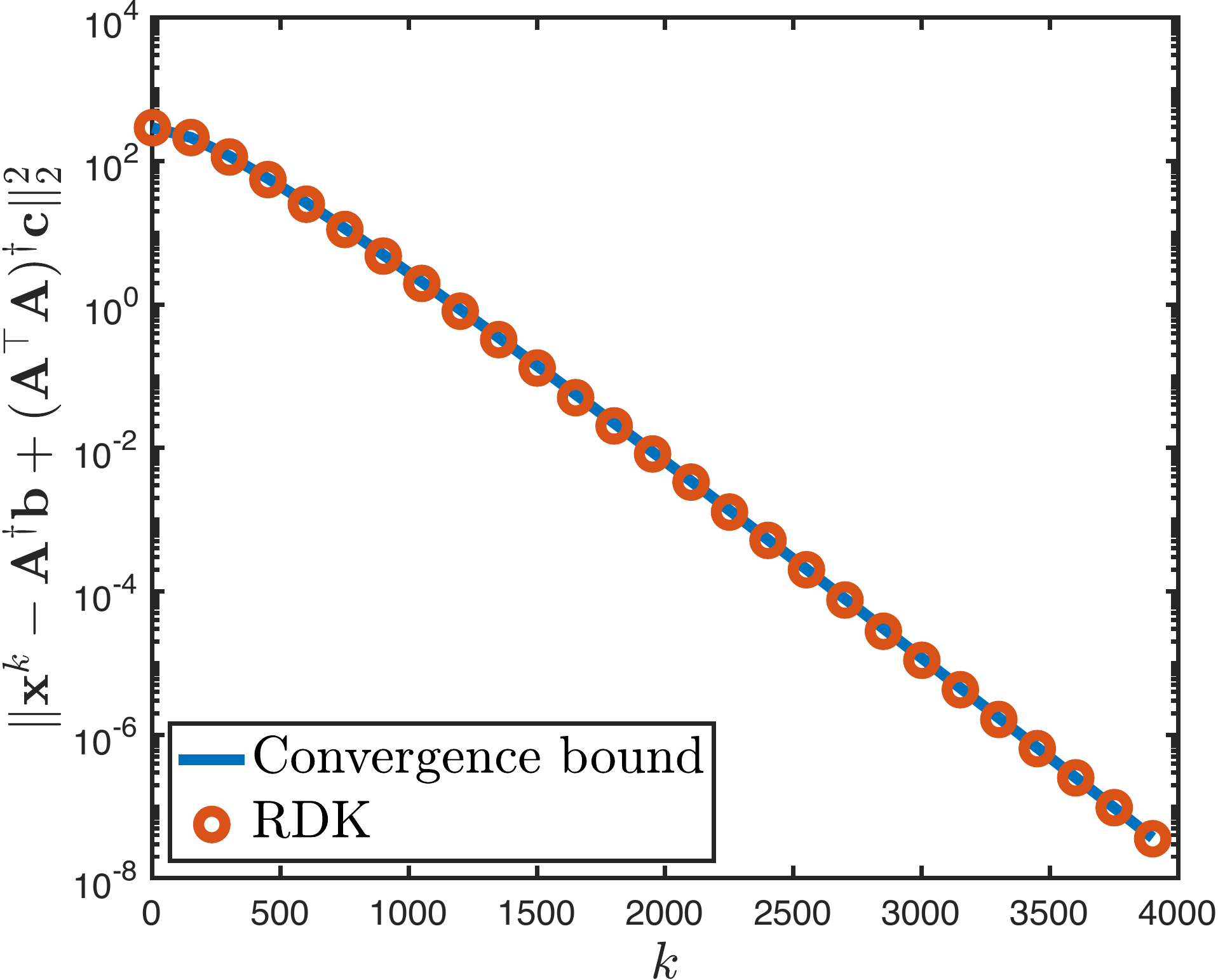,height=2.4in}\quad\epsfig{figure=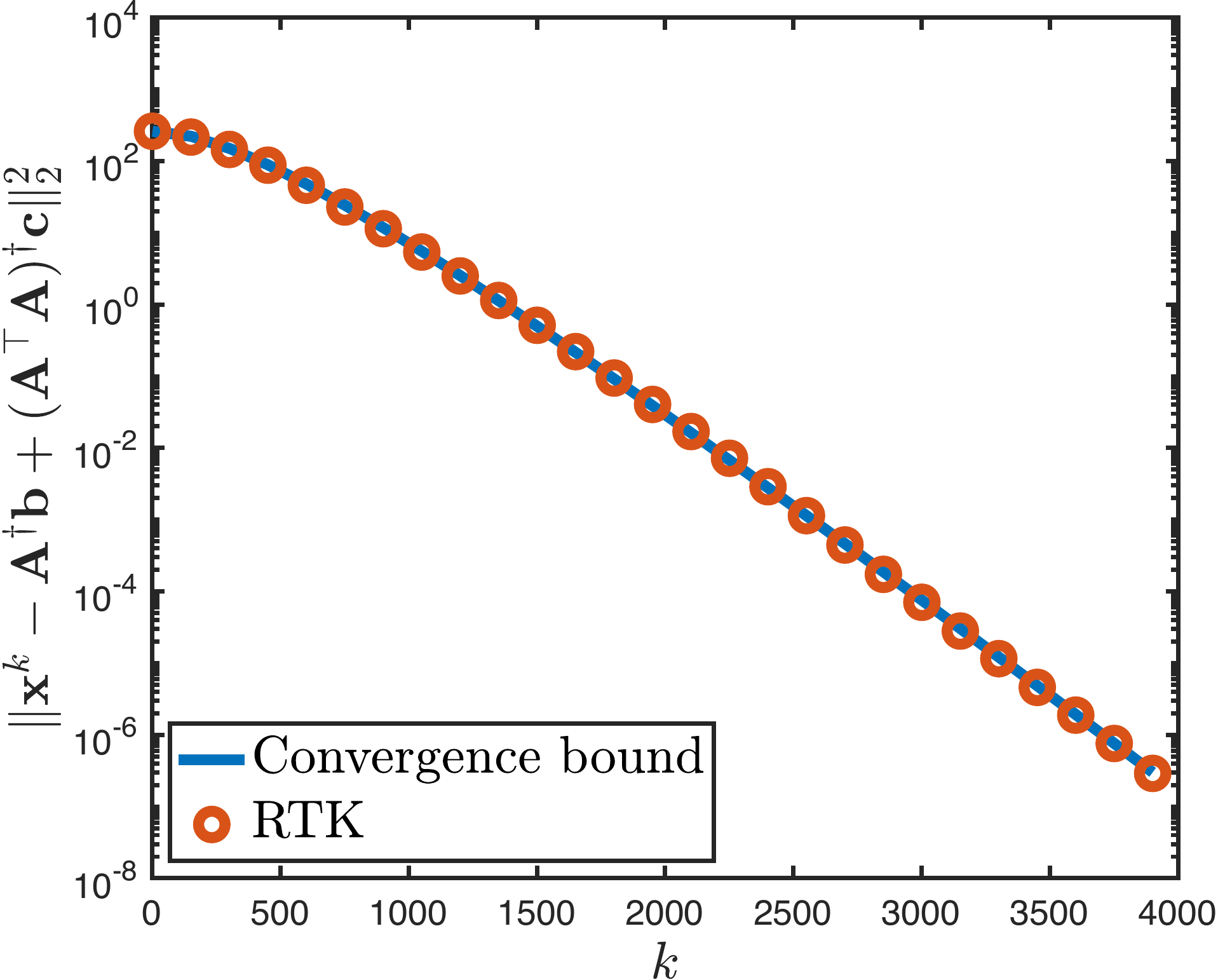,height=2.4in}}
\caption{The error $\|{\bf x}^k-{\bf A}^\dag{\bf b}+{(\bf A^\top A)^\dag{\bf c}}\|_2^2$ (average of 50 independent trials) for $m=500$, $n=250$, $r=150$, and $\kappa=1$. Left: RDK for the case $\mbf c\in\ran(\mbf A^\top)$. Right: RTK for the case $\mbf c\notin\ran(\mbf A^\top)$.}
\label{fig1}
\end{figure}
\begin{figure}[htb]
\centerline{\epsfig{figure=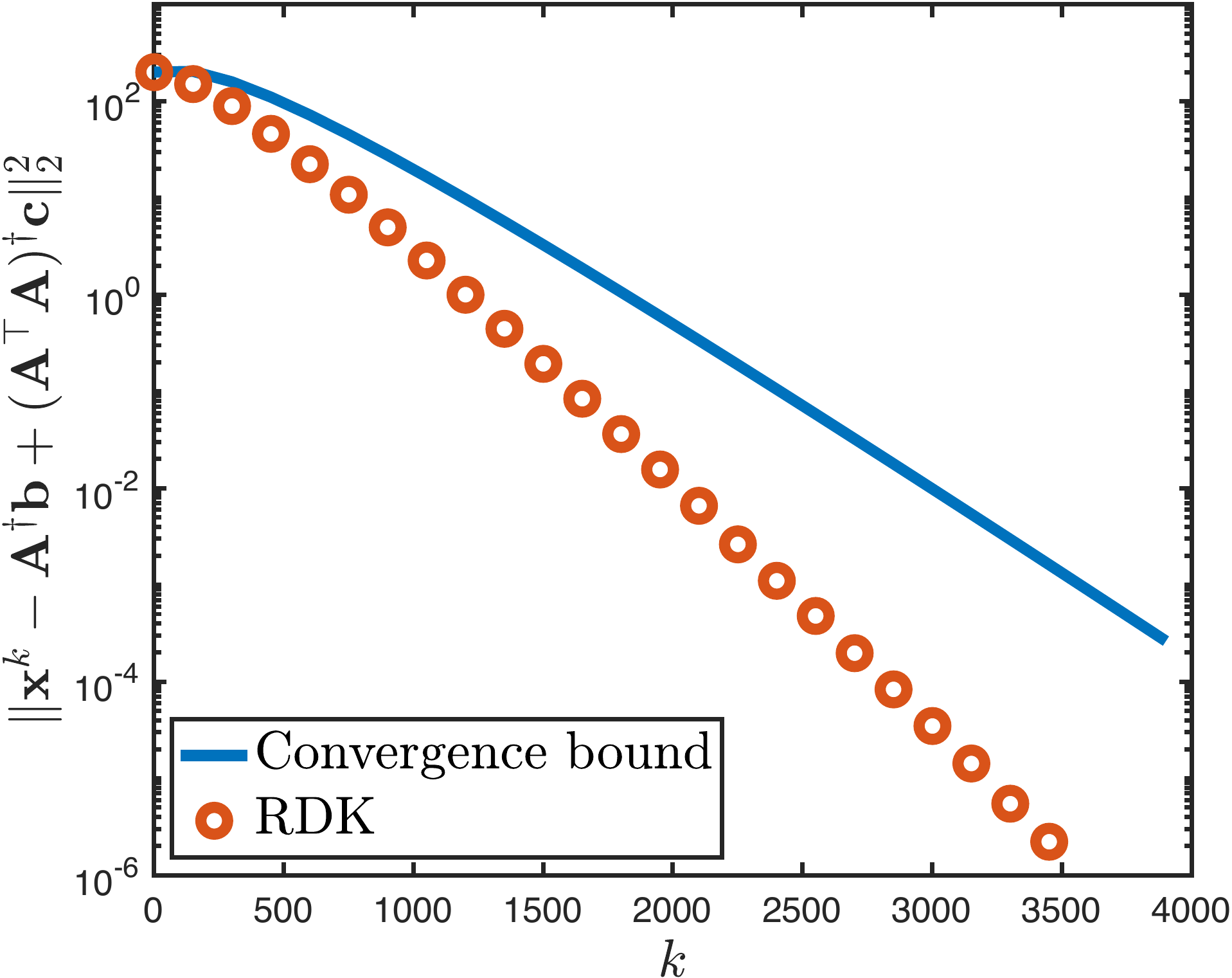,height=2.4in}\quad\epsfig{figure=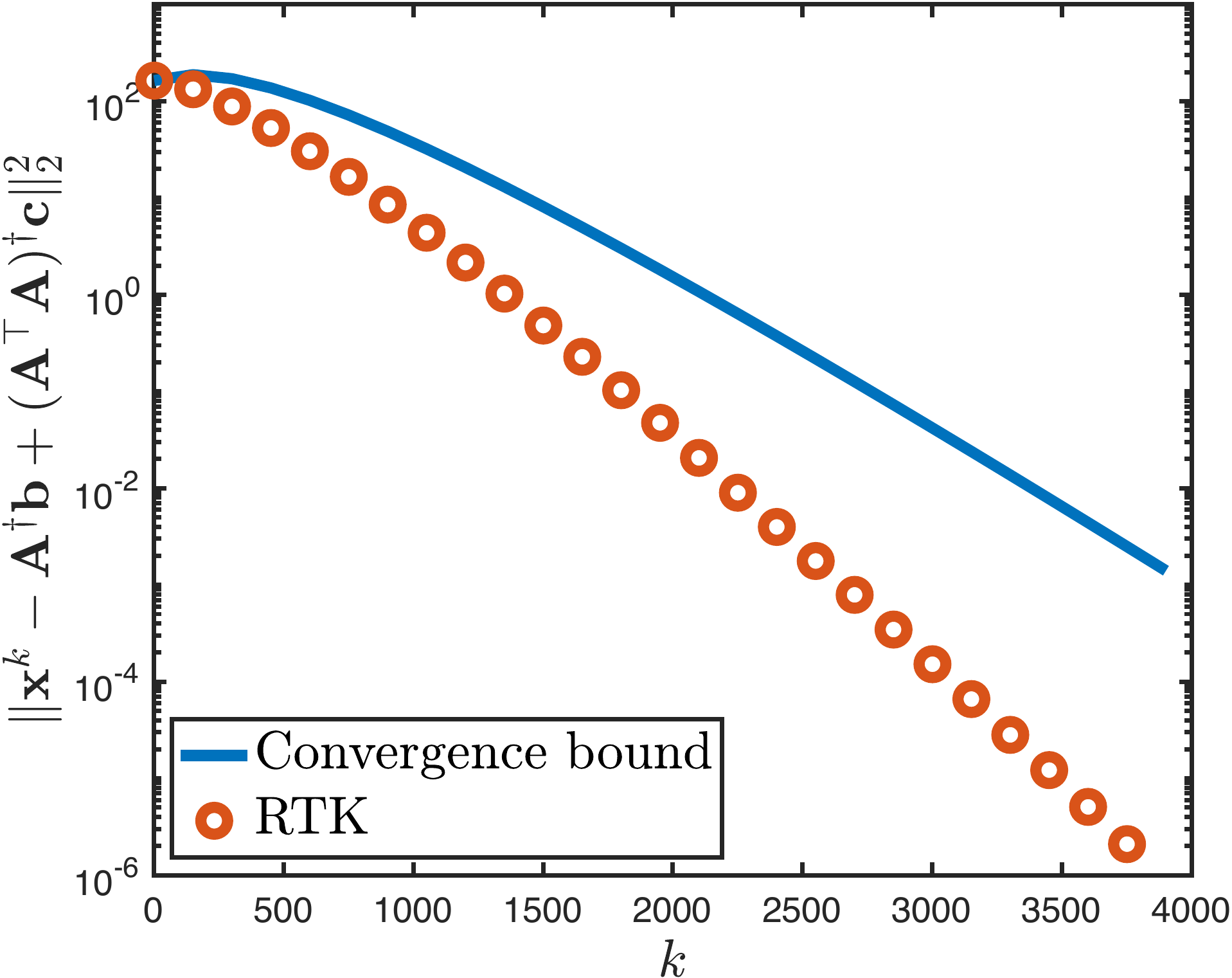,height=2.4in}}
\caption{The error $\|{\bf x}^k-{\bf A}^\dag{\bf b}+{(\bf A^\top A)^\dag{\bf c}}\|_2^2$ (average of 50 independent trials) for $m=500$, $n=250$, $r=150$, and $\kappa=1.5$. Left: RDK for the case $\mbf c\in\ran(\mbf A^\top)$. Right: RTK for the case $\mbf c\notin\ran(\mbf A^\top)$.}
\label{fig2}
\end{figure}

\section{Concluding remarks}
In this work, we propose randomized iterative algorithms that solve the extended normal equations. We prove that the RDK algorithm exponentially converges to a solution of the extended normal equations for the  consistent case and prove that the RTK algorithm exponentially converges to a least squares solution of the extended normal equations for the inconsistent case. Our convergence analysis applies to arbitrary matrix $\mbf A$ and the convergence upper bounds are attained for the case that all nonzero singular values of $\mbf A$ are the same. Numerical experiments confirm the theoretical results. 

We remark that for the scenarios where $\mbf A$ is so large that it does not fit in computer memory, iterative methods based on {\it full} matrix-vector multiplications (e.g., Krylov subspace methods) are inefficient because the entire matrix $\bf A$ must be accessed in each step (which leads huge communication costs). If memory is a concern, the proposed RDK and RTK algorithms are appropriate alternatives because at each step only a sample of rows and columns are required.

\section*{Acknowledgments}
This work was funded by the National Natural Science Foundation of China (No.11771364) and the Fundamental Research Funds for the Central Universities (No.20720180008).


\begin{thebibliography}{10}

\bibitem{bai2018greed}
Z.-Z. Bai and W.-T. Wu.
\newblock On greedy randomized {K}aczmarz method for solving large sparse
  linear systems.
\newblock {\em SIAM J. Sci. Comput.}, 40(1):A592--A606, 2018.

\bibitem{bai2018relax}
Z.-Z. Bai and W.-T. Wu.
\newblock On relaxed greedy randomized {K}aczmarz methods for solving large
  sparse linear systems.
\newblock {\em Appl. Math. Lett.}, 83:21--26, 2018.

\bibitem{bai2019parti}
Z.-Z. Bai and W.-T. Wu.
\newblock On partially randomized extended {K}aczmarz method for solving large
  sparse overdetermined inconsistent linear systems.
\newblock {\em Linear Algebra Appl.}, 578:225--250, 2019.

\bibitem{calandra2019appro}
H.~Calandra, S.~Gratton, E.~Riccietti, and X.~Vasseur.
\newblock On the approximation of the solution of partial differential
  equations by artificial neural networks trained by a multilevel
  {L}evenberg-{M}arquardt method.
\newblock {\em arXiv preprint arXiv:1904.04685}, 2019.

\bibitem{calandra2020itera}
H.~Calandra, S.~Gratton, E.~Riccietti, and X.~Vasseur.
\newblock On iterative solution of the extended normal equations.
\newblock {\em SIAM J. Matrix Anal. Appl.}, 41(4):1571--1589, 2020.

\bibitem{du2019tight}
K.~Du.
\newblock Tight upper bounds for the convergence of the randomized extended
  {K}aczmarz and {G}auss-{S}eidel algorithms.
\newblock {\em Numer. Linear Algebra Appl.}, 26(3):e2233, 14, 2019.

\bibitem{du2020rando}
K.~Du, W.-T. Si, and X.-H. Sun.
\newblock Randomized extended average block {K}aczmarz for solving least
  squares.
\newblock {\em SIAM J. Sci. Comput.}, accepted, 2020.

\bibitem{fletcher1972class}
R.~Fletcher.
\newblock A class of methods for non-linear programming. {III}. {R}ates of
  convergence.
\newblock In {\em Numerical methods for non-linear optimization ({C}onf.,
  {D}undee, 1971)}, pages 371--381. 1972.

\bibitem{leventhal2010rando}
D.~Leventhal and A.~S. Lewis.
\newblock Randomized methods for linear constraints: convergence rates and
  conditioning.
\newblock {\em Math. Oper. Res.}, 35(3):641--654, 2010.

\bibitem{ma2015conve}
A.~Ma, D.~Needell, and A.~Ramdas.
\newblock Convergence properties of the randomized extended {G}auss-{S}eidel
  and {K}aczmarz methods.
\newblock {\em SIAM J. Matrix Anal. Appl.}, 36(4):1590--1604, 2015.

\bibitem{ma2018itera}
A.~Ma, D.~Needell, and A.~Ramdas.
\newblock Iterative methods for solving factorized linear systems.
\newblock {\em SIAM J. Matrix Anal. Appl.}, 39(1):104--122, 2018.

\bibitem{moorman2020rando}
J.~D. Moorman, T.~K. Tu, D.~Molitor, and D.~Needell.
\newblock Randomized {K}aczmarz with averaging.
\newblock {\em BIT}, to appear, 2020.

\bibitem{necoara2019faste}
I.~Necoara.
\newblock Faster randomized block {K}aczmarz algorithms.
\newblock {\em SIAM J. Matrix Anal. Appl.}, 40(4):1425--1452, 2019.

\bibitem{needell2015rando}
D.~Needell, R.~Zhao, and A.~Zouzias.
\newblock Randomized block {K}aczmarz method with projection for solving least
  squares.
\newblock {\em Linear Algebra Appl.}, 484:322--343, 2015.

\bibitem{niu2020greed}
Y.-Q. Niu and B.~Zheng.
\newblock A greedy block {K}aczmarz algorithm for solving large-scale linear
  systems.
\newblock {\em Appl. Math. Lett.}, 104:106294, 8, 2020.

\bibitem{strohmer2009rando}
T.~Strohmer and R.~Vershynin.
\newblock A randomized {K}aczmarz algorithm with exponential convergence.
\newblock {\em J. Fourier Anal. Appl.}, 15(2):262--278, 2009.

\bibitem{wu2020proje}
N.~Wu and H.~Xiang.
\newblock Projected randomized {K}aczmarz methods.
\newblock {\em J. Comput. Appl. Math.}, 372:112672, 2020.

\bibitem{zhang2020relax}
J.~Zhang and J.~Guo.
\newblock On relaxed greedy randomized coordinate descent methods for solving
  large linear least-squares problems.
\newblock {\em Appl. Numer. Math.}, 157:372--384, 2020.

\bibitem{zhang2019new}
J.-J. Zhang.
\newblock A new greedy {K}aczmarz algorithm for the solution of very large
  linear systems.
\newblock {\em Appl. Math. Lett.}, 91:207--212, 2019.

\bibitem{zouzias2013rando}
A.~Zouzias and N.~M. Freris.
\newblock Randomized extended {K}aczmarz for solving least squares.
\newblock {\em SIAM J. Matrix Anal. Appl.}, 34(2):773--793, 2013.

\end{thebibliography}
\end{document}